\documentclass[letterpaper]{amsart}
 
\usepackage{amsfonts}
\usepackage{amsmath}
\usepackage{amssymb}
\usepackage{amsthm}
\usepackage{graphicx}
\usepackage{mathrsfs}
\usepackage{colonequals}
\usepackage{units}
\usepackage{euscript}

\usepackage[margin=1.50in]{geometry}

\usepackage{url}
\usepackage{xcolor}
\usepackage{hyperref}
\usepackage[nameinlink]{cleveref}
\usepackage{comment}

\theoremstyle{plain}
\newtheorem{theorem}{Theorem}[section]
\newtheorem{lemma}[theorem]{Lemma}
\newtheorem*{appendixlemma}{Lemma A.1}
\newtheorem{corollary}[theorem]{Corollary}
\newtheorem{proposition}[theorem]{Proposition}

\crefname{claim}{Claim}{Claims}
\newtheorem*{claim*}{Claim}
\newtheorem{remark}[theorem]{Remark}

\theoremstyle{definition}
\newtheorem{definition}[theorem]{Definition}
\newtheorem{question}[theorem]{Question}
\crefname{question}{Question}{Questions}

\newtheorem*{convention*}{Convention}
\newtheorem*{remark*}{Remark}

\theoremstyle{remark}

\makeatletter
\let\c@equation\c@theorem
\makeatother
\numberwithin{equation}{section}

\newcommand{\R}{\mathbb{R}}
\newcommand{\Z}{\mathbb{Z}}

\newcommand{\N}{\mathbb{N}}

\newcommand{\scl}{\mathrm{scl}}
\newcommand{\cl}{\mathrm{cl}}
\newcommand{\Map}{\mathrm{Map}}
\newcommand{\supp}{\mathrm{supp}}

\title{Stable commutator length on big mapping class groups}

\author{Elizabeth Field, Priyam Patel, and Alexander J. Rasmussen}

\begin{document}
\maketitle

\begin{abstract}
    
    We study stable commutator length on mapping class groups of certain infinite-type surfaces. In particular, we show that stable commutator length defines a continuous function on the commutator subgroups of such infinite-type mapping class groups. We furthermore show that the commutator subgroups are open and closed subgroups and that the abelianizations are finitely generated in many cases. Our results apply to many popular infinite-type surfaces with locally coarsely bounded mapping class groups. 
\end{abstract}

\section{Introduction}

Infinite-type surfaces and their mapping class groups have generated intense interest in the last several years. Two of the main themes in this area have been (1) the study of bounded cohomology of mapping class groups of infinite-type surfaces (so-called \emph{big mapping class groups}), and (2) the study of abelianizations of big mapping class groups. We address both of these themes in the present note. For the following discussion and the rest of the paper, a surface will be taken to mean an \emph{orientable surface without boundary}.

The recent interest in big mapping class groups may be largely attributed to a blog post of Calegari \cite{calblog} from 2009. The two themes mentioned above figure heavily in Calegari's original blog post, where he proves that the mapping class group of the sphere minus a Cantor set is \emph{uniformly perfect} (that is, every element is a product of a bounded number of commutators) and asks whether the second bounded cohomology of the mapping class group of the plane minus a Cantor set is trivial. Calegari's question was answered by J. Bavard in \cite{ray}, where she proves that the second bounded cohomology of the mapping class group of the plane minus a Cantor set is in fact infinite dimensional. The study of the second bounded cohomology of big mapping class groups has recently been taken considerably further (see \cite{AbbottMillerPatel}, \cite{simultaneous}, \cite{hypactions}, and \cite{wwpd}). In particular, if $S$ has a \emph{non-displaceable subsurface}, in the terminology of Mann--Rafi \cite{ManRaf20}, then the second bounded cohomology of its mapping class group, $H_b^2(\Map(S);\R)$, is known to be infinite dimensional (see \cite{hypactions}). 

While the definition of bounded cohomology is technical, it is a group invariant which has numerous applications. Some of the most important applications are those regarding rigidity properties of groups.
For instance, Bestvina--Fujiwara use bounded cohomology in \cite{bf} to prove the Farb--Kaimanovich--Masur Rigidity Theorem that higher rank lattices do not embed in mapping class groups of finite-type surfaces. In the context of big mapping class groups, bounded cohomology may be useful in producing obstructions for groups to act on finite-type surfaces by homeomorphisms (see \cite{ray} and \cite{circular}).

Second bounded cohomology is frequently understood using quasimorphisms. If $G$ is a group, then a function $\phi:G\to \R$ is a \emph{quasimorphism} if the quantity $|\phi(gh)-\phi(g)-\phi(h)|$ is bounded uniformly over all $g,h\in G$. The space of non-trivial quasimorphisms of $G$ embeds into the second bounded cohomology of $G$ with real coefficients. Thus, producing an infinite dimensional space of quasimorphisms is a frequently-used method for demonstrating infinite dimensionality of the second bounded cohomology of a group. Quasimorphisms may in turn be understood using the \emph{stable commutator length (or scl) function}, $\scl: [G,G] \to \R_{\geq 0}$. Namely, by C. Bavard's Duality Theorem \cite{scl}, $\scl(g)>0$ exactly when there is a quasimorphism $\phi:G\to \R$ which is unbounded on the powers of $g$ and infinite $\ell^\infty$ distance from any homomorphism. Given this reformulation, the result of J. Bavard mentioned above amounts to saying that the  $\scl$ function does not vanish on the mapping class group of the plane minus a Cantor set.

In this paper, we study the stable commutator length function on $\Map(S)$ for a broader class of infinite-type surfaces. Since $\scl$ is generally only defined on the commutator subgroup, it is important to understand the properties of this subgroup as well as its size as a subgroup of the mapping class group, as measured by studying the abelianization $\Map(S)/[\Map(S),\Map(S)]$. 
Using the classification of infinite-type surfaces with \emph{locally coarsely bounded} mapping class groups and the partial order on the space of ends developed by Mann--Rafi in \cite{ManRaf20}, we state our first main result as follows.

\newcommand{\tSCLisContinuous}{Let $S$ be an infinite-type surface. Assume that $\Map(S)$ is locally coarsely bounded and that every equivalence class of maximal ends of $S$ is infinite, except (possibly) for finitely many isolated planar ends. Then, stable commutator length is a continuous function on $[\Map(S), \Map(S)]$. Furthermore, the commutator subgroup is both an open and closed subgroup of $\Map(S)$ and is, thus, also Polish.
}

\begin{theorem} \label{T:SCLContinuous}
\tSCLisContinuous
\end{theorem}

\noindent The locally coarsely bounded property (called locally CB for short) is related to the actions of $\Map(S)$ on metric spaces and is relevant to defining a coarse geometry for $\Map(S)$. The maximal ends of $S$ are those ends which are maximal with respect to the partial order on the ends of $S$ mentioned above; see \Cref{S: CBEnds} for more details. These maximal ends are organized into equivalence classes of ends with the same topology, and each equivalence class consists of either a finite set or a Cantor set of ends. We note that the hypotheses in \Cref{T:SCLContinuous} are satisfied by an uncountably infinite collection of surfaces and apply to many popular infinite-type surfaces, including the plane minus a Cantor set, the once-punctured blooming Cantor tree surface, and any finite-type surface minus a Cantor set. The commutator subgroup being closed under the above hypotheses is a bit surprising, since the results of Dickmann--Domat \cite[Theorem A.1]{Domat} and Malestein--Tao \cite[Theorem 5.3]{selfsim} show that the commutator subgroup is \emph{not} closed for a large class of infinite-type surfaces (including the Loch Ness Monster surface). 

\Cref{T:SCLContinuous} implies that whenever $\scl(g)$ is positive, it is also positive for all nearby mapping classes $h$. That is, by modifying $g$ outside of any sufficiently large compact subsurface of $S$, we obtain a new mapping class with positive stable commutator length. This provides numerous new examples of mapping classes with positive stable commutator length. On the other hand, J. Bavard shows in \cite{ray} that there are infinite-type mapping classes which act loxodromically on natural hyperbolic graphs, but which nonetheless have zero $\scl$. It is interesting to compare this to the fact that for finite-type surfaces, most pseudo-Anosovs have positive scl \cite{BesBroFuj16}, and pseudo-Anosovs are exactly the mapping classes that act as loxodromic isometries on the curve graph. In view of these results, we ask the following question.

\begin{question}
Let $S$ be an infinite-type surface for which $\Map(S)$ is locally CB and such that every equivalence class of maximal ends of $S$ is infinite, except for finitely many isolated planar ends. Let $Z = \{g\in \Map(S) \mid \scl(g) = 0\}$. What are the topological properties of the subset $Z$? In particular, is $Z$ an open subset of $\Map(S)$?
\end{question}

\noindent Note that by \Cref{T:SCLContinuous}, the set $Z$ must be a closed subset of $\Map(S)$. 
Progress on this question will be crucial to better understanding the stable commutator length function on big mapping class groups.

The proof of \Cref{T:SCLContinuous} relies on finding a neighborhood of the identity in $\Map(S)$ on which scl is constantly zero. That such a neighborhood exists is the content of our next main result. We show that there is a uniform bound on commutator length (denoted $\cl$) for this neighborhood, and the neighborhood depends on a particular partitioning of the end space of $S$ from the classification of surfaces with locally CB mapping class groups \cite{ManRaf20}. In particular, Mann--Rafi show that if $S$ is an infinite-type surface such that $\Map(S)$ is locally CB, then there exists a finite-type subsurface $K\subset S$ such that the complementary components of $K$ partition $E(S)$ into finitely many clopen sets \[E(S) = A_1 \sqcup \cdots \sqcup A_n \sqcup P_1 \sqcup \cdots \sqcup P_m, \] where each $A_i$ is \emph{self-similar} and each $P_k$ is homeomorphic to a clopen subset of some $A_i$. Given this partition, we prove the following result.

\newcommand{\tBoundedCommLength}{Let $S$ be an infinite-type surface. Assume that $\Map(S)$ is locally coarsely bounded and that every equivalence class of maximal ends of $S$ is infinite, except (possibly) for finitely many isolated planar ends. Then, there exists an open subgroup $\mathcal{O}$ of $\Map(S)$ such that the function $\cl$ is bounded from above by 2 on $\mathcal{O}$.
}

\begin{theorem} \label{T:BoundedCommutatorLength}
\tBoundedCommLength
\end{theorem}

\noindent The set $ \mathcal O$ in the above result is taken to be a neighborhood of the identity of the form $\mathcal{V}_{K'} = \{g\in \Map(S) \mid g|_{K'} = \operatorname{id}\}$ for some $K' \supset K$. 
Another consequence of the existence of this neighborhood of the identity is the following result.

\newcommand{\tFiniteAbelianization}{
Let $S$ be a surface with tame end space, such that $\Map(S)$ is CB generated and locally (but not globally) CB, and such that every equivalence class of maximal ends of $S$ is infinite, except (possibly) for finitely many isolated planar ends. Then, the abelianization $\Map(S)/[\Map(S),\Map(S)]$ is finitely generated and is discrete when endowed with the quotient topology.
}
\begin{theorem}\label{T:FiniteAbelianization}
\tFiniteAbelianization
\end{theorem}

Our method for proving \Cref{T:BoundedCommutatorLength} utilizes a technique which originated in Calegari's blog post \cite{calblog}. This technique has a long history for proving perfectness of various groups of homeomorphisms. 
The main tool used in this paper was independently and concurrently developed by Malestein--Tao in \cite{selfsim} to prove uniform perfectness and perfectness of certain big mapping class groups; compare our \Cref{L:SelfSimilarSubsetsHomeo}, \Cref{ShiftMapTake2}, and \Cref{T:BoundedCommutatorLength} with their Lemma 2.8, Lemma 3.5, and Lemma 3.7, respectively. In the case where the surface $S$ is ``uniformly self-similar,'' meaning that $S$ has 0 or infinite genus and that the entire end space is self-similar with exactly one equivalence class of maximal ends which is infinite, their results imply ours trivially. However, the uniformly self-similar surfaces are a small subset of the set of surfaces to which our results apply. In particular, within the class of surfaces we consider, the results of Malestein--Tao apply to those surfaces which are globally CB. We also note that the surfaces covered in \Cref{T:FiniteAbelianization} include finite-type surfaces with a Cantor set of points removed. The abelianization of the mapping class group of these surfaces has been previously studied and is known by \cite{CalChen}.

While the conditions we require the surfaces in \Cref{T:SCLContinuous} and \Cref{T:BoundedCommutatorLength} to satisfy are somewhat difficult to parse, it turns out that there is a nice topological characterization of such surfaces which was suggested to us by J. Malestein and J. Tao. In particular, the surfaces to which these theorems apply are those which can be constructed as a connected sum of uniformly self-similar surfaces with a finite-type surface. 

\newcommand{\tTopologicalCharacterization}{
A surface $S$ is the connected sum of finitely many uniformly self-similar surfaces with a finite-type surface if and only if $\Map(S)$ is locally CB and each equivalence class of maximal ends is infinite, except (possibly) for finitely many isolated planar ends.
}

\begin{appendixlemma}\label{T:TopologicalCharacterization}
\tTopologicalCharacterization
\end{appendixlemma}

\medskip

\noindent \textbf{Outline of paper.} In \Cref{S:Background}, we provide background on stable commutator length and quasimorphisms, as well as relevant details about infinite-type surfaces and their mapping class groups. In \Cref{S:ProofOf1.3}, we give the proof of \Cref{T:BoundedCommutatorLength}. In \Cref{S:ProofOf1.1}, we provide the proofs of \Cref{T:SCLContinuous} and \Cref{T:FiniteAbelianization}. Finally, in the Appendix, we provide a concrete topological characterization of the surfaces which satisfy the hypotheses of \Cref{T:SCLContinuous} and \Cref{T:BoundedCommutatorLength}.

\medskip

\noindent \textbf{Acknowledgements.} The authors thank Justin Malestein and Jing Tao for useful conversations about their paper \cite{selfsim} and for suggesting the topological characterization of surfaces which satisfy the hypotheses of \Cref{T:SCLContinuous} and \Cref{T:BoundedCommutatorLength}. The authors also thank Kathryn Mann and Kasra Rafi for useful conversations about their paper \cite{ManRaf20}. Lastly, we thank the referee for thoughtful comments and suggestions that have increased the quality of this paper. Field was partially supported by NSF grants DMS-1840190 and DMS-2103275. Patel was partially supported on DMS-1937969 and DMS-2046889. Rasmussen was partially supported by NSF grant DMS-1840190.

\section{Preliminaries}\label{S:Background}
\subsection{Stable commutator length and quasimorphisms}

\begin{definition}
Let $G$ be a group and let $g\in [G,G]$ be an element of the commutator subgroup. The \emph{commutator length} of $g$ is defined to be $$\cl(g) = \inf \{n \mid g = [h_1, k_1][h_2, k_2] \cdots [h_n, k_n]\},$$ where $[h,k]$ denotes the commutator $hkh^{-1}k^{-1}$.
\end{definition}

\noindent Note that commutator length is sub-additive, i.e. $\cl(g^m\cdot g^n) \leq \cl(g^m) + \cl(g^n)$. In particular, $\cl(g^n) \leq n\cdot\cl(g)$ for all $n\geq 1$. 

\begin{definition} 
The \emph{stable commutator length} of $g\in [G, G]$ is defined to be $$\scl(g) = \lim_{n\to \infty} \frac{\cl(g^n)}{n}.$$ By the sub-additivity of commutator length, we see that this limit exists and that $\scl(g) \leq \cl(g)$.
\end{definition}

\noindent Thus, stable commutator length is a function $\scl:[G,G]\to \R_{\geq 0}$ that measures the rate of growth of commutators in $g^n$.
The notion of stable commutator length is intimately connected with quasimorphisms and bounded cohomology. In particular, bounded cohomology is understood using quasimorphisms, and quasimorphisms may in turn be understood using the stable commutator length function. A map $\phi: G\to \mathbb{R}$ is a \emph{quasimorphism} if there exists a constant $D$ such that $|\phi(gh)-\phi(g)-\phi(h)|\leq D$ for all $g,h\in G$. 

\begin{definition}
A quasimorphism $\phi: G\to \mathbb{R}$ is said to have \emph{defect at most $D$} if for all $g,h\in G$, $$|\phi(gh) - \phi(g) - \phi(h)| \leq D.$$ We denote by $D(\phi)$ the smallest $D$ such that $\phi$ has defect at most $D$ and call $D(\phi)$ the \emph{defect} of $\phi$. Note that $\phi$ is a homomorphism if and only if $D(\phi) = 0$. We say that $\phi$ is \emph{homogeneous} if for all $g\in G$ and $n\in \mathbb{Z}$, $\phi(g^n) = n\phi(g)$. 
\end{definition}

We let $\widetilde{QH}(G)$ denote the set of homogeneous quasimorphisms from $G$ to $\mathbb{R}$. The set $\widetilde{QH}(G)$ is a real vector space, and the homomorphisms from $G$ to $\R$, $H^1(G;\R)$, make up a linear subspace of $\widetilde{QH}(G)$. Let $\widetilde{QH}_1(G)$ denote the subset of $\widetilde{QH}(G)$ consisting of homogeneous quasimorphisms of defect one.
The connection between quasimorphisms and stable commutator length is best summarized by the following theorem (see \cite{scl} or \cite[Section 2.5]{calbook} for more details).

\begin{theorem}[Bavard duality, \cite{scl}]\label{BavardDuality}
For any $g\in [G,G]$, $$\scl(g) = \sup_{\phi\in \widetilde{QH}(G) \setminus H^1(G;\R)} \frac{|\phi(g)|}{2D(\phi)} = \sup_{\phi\in \widetilde{QH}_1(G)} \frac{|\phi(g)|}{2}.$$
\end{theorem}

\noindent An important consequence of this duality is that $\scl(g) = 0$ if and only if $\phi(g) = 0$ for all non-trivial homogeneous quasimorphisms $\phi$.

The \emph{coboundary} is a linear map from the vector space $\widetilde{QH}(G)$ to the second bounded cohomology $H_b^2(G;\R)$. Coboundaries vanish exactly on the subspace $H^1(G;\R)$, and the quotient $\widetilde{QH}(G)/H^1(G;\R)$ is identified with a linear subspace of $H_b^2(G;\R)$. This subspace turns out to coincide exactly with the kernel of the \emph{forgetful map} $H_b^2(G;\R)\to H^2(G;\R)$ (see \cite{ghys} for additional details). Therefore, if a group $G$ has an infinite-dimensional space of nontrivial homogeneous quasimorphisms, then the second bounded cohomology of $G$ is also infinite-dimensional.

\subsection{Infinite-type surfaces}\label{S:ends}
Let $S$ be an infinite-type surface, i.e., a surface whose fundamental group is not finitely generated.

\begin{definition}
An \emph{exiting sequence} in \( S \) is a sequence \( \{U_n\}_{n\in\N} \) of connected, open subsets of \( S \) satisfying:
\begin{enumerate}
\item \( U_{n} \subset U_m \) whenever $m<n$,
\item $U_n$ is not relatively compact for any $n \in \N$,
\item the boundary of \(U_n \) is compact for each \( n \in \N \), and
\item any relatively compact subset of $S$ is disjoint from all but finitely many of the $U_n$’s.
\end{enumerate}
Two exiting sequences \( \{U_n\}\) and \( \{V_n\} \) are \emph{equivalent} if for every \( n \in \N \) there exists \( m \in \N \) such that \( U_m \subset V_n \) and \( V_m \subset U_n \). We denote the equivalence class of an exiting sequence by $[\{U_n\}]$ and call such an equivalence class an \emph{end} of \( S \). An end is \emph{planar} if there exists an $i$ such that $U_i$ is homeomorphic to an open subset of $\mathbb{R}^2$ and is \emph{nonplanar}, or \emph{accumulated by genus}, if every $U_i$ has infinite genus.

\end{definition}

The space of ends of $S$, $E = E(S)$, is equipped with a natural topology for which $E$ is a totally disconnected, compact, and second countable topological space (and thus a closed subset of the Cantor set; see \cite[Proposition~5]{Ric63}). By work of Richards \cite{Ric63}, an orientable, boundaryless, infinite-type surface, $S$, is completely classified by its (possibly infinite) genus, its space of ends, $E = E(S)$, and the (closed) subset of ends which are accumulated by genus, $E^G = E^G(S)$.  A pair $(E,E^G)$, where $E$ is a closed subset of the Cantor set and $E^G$ is a closed subset of $E$, is called an \emph{end space} and can be realized as the space of ends of some surface \cite[Theorem~2]{Ric63}.

To describe the topology on $E$, let \( V \) be an open subset of \( S \) with compact boundary, and define \[ \widehat V = \left\{ \left[\{U_n\}\right]\in E \mid U_m \subset V \text{ for some } m\in \N\right\}.\] Let \( \mathcal V = \{ \widehat V \mid V \subset S \text{ is open with compact boundary}\} \). 
The set \( E \) becomes a topological space by declaring \( \mathcal V \) a basis for the topology. 
Note that the definition of this topology implies that if a subset $E' \subset E$ is clopen, then there exists a simple closed curve $\gamma$ such that $S\setminus\gamma$ has two connected components, one whose end space is exactly $E'$ and one whose end space is $E\setminus E'$. 

\begin{convention*}
We will denote subsets of ends simply by capital roman letters, like $A$, which will always refer to the \emph{pair} of subspaces $(A,A\cap E^G)$ of the pair $(E,E^G)$. Two such subspaces $A$ (identified with $(A,A\cap E^G)$) and $B$ (identified with $(B,B\cap E^G)$) will be considered to be homeomorphic exactly when the pairs $(A,A\cap E^G)$ and $(B,B\cap E^G)$ are homeomorphic. We also use the convention that a \emph{neighborhood} of a point in $E$ will always refer to a \emph{clopen} neighborhood in $E$. Lastly, when we refer to a component of $S\setminus\gamma$ we mean its closure in $S$ in order to avoid introducing a new end in the component. 
\end{convention*}

\subsection{Coarse boundedness and self-similar end spaces}\label{S: CBEnds}

Given any surface $S$, the \emph{mapping class group} of $S$, $\Map(S)$, is the group of orientation-preserving homeomorphisms of $S$ up to isotopy. For a finite-type surface, the mapping class group is finitely generated, while the mapping class group of an infinite-type surface is not. In either setting, the mapping class group is equipped with the induced topology from the compact-open topology of the space of homeomorphisms of $S$. With this topology, the mapping class group of an infinite-type surface, or big mapping class group, is a Polish topological group but is neither locally compact nor compactly generated (see \cite[Proposition~12.4.1 and Theorem~12.4.2]{AraVla20}). This fact gives one of the connections between big mapping class groups and descriptive set theory and has been used to prove a variety of results about these groups.

We now record a fact about topological groups which will also be used to prove \Cref{T:SCLContinuous}.

\begin{lemma}\label{L:OpenClosedSubgroup}
Let $G$ be a topological group and let $H$ be a subgroup of $G$. If $H$ contains a subset $U$ which is open in $G$, then $H$ is both open and closed in $G$.
\end{lemma}

\begin{proof}
As $H$ is a subgroup of $G$ and $U$ is a subset of $H$, we have that $$H = HU = \bigcup_{h\in H}hU.$$ Since we can write $H$ as a union of open sets, $H$ is open in $G$. Now, note that the complement $G \setminus H$ is a union of cosets of $H$. Since each coset of $H$ is open in $G$, this implies that $G \setminus H$ is open, and hence $H$ is also closed in $G$.
\end{proof}

Many of the standard tools of geometric group theory were originally designed for studying finitely generated groups, and these tools have natural analogs in the setting of topological groups that are locally compact and compactly generated. However, these analogs are still insufficient for understanding the coarse geometry of big mapping class groups. In \cite{Ros14}, Rosendal introduces the following geometric notion which is weaker than compactness and which was used by Mann--Rafi in \cite{ManRaf20} to study the large-scale geometry of certain big mapping class groups.

\begin{definition}[Coarsely bounded]
Let $G$ be a Polish topological group. We say that a subset $A\subseteq G$ is \emph{coarsely bounded}, or \emph{CB}, in $G$ if $A$ has finite diameter in every compatible left-invariant metric on $G$. $G$ is \emph{locally coarsely bounded}, or \emph{locally CB}, if there exists an open neighborhood of the identity in $G$ which is CB, and $G$ is \emph{CB generated} if $G$ is algebraically generated by a $CB$ subset.
\end{definition}

\noindent A key utility of this generalization of compactness is that Polish groups which are locally CB and which have a CB generating set have a well-defined quasi-isometry type (see \cite{Ros14}). 

Mann--Rafi give a classification of surfaces whose mapping class groups are locally CB in \cite{ManRaf20}. A key tool in their classification is the following partial order on the space of ends that captures the local complexity of an end (see \cite[Section~4]{ManRaf20}).

\begin{definition}
Given an end space $(E, E^G)$, define a preorder $\preccurlyeq$ on $E$ as follows. For $x,y\in E$, we say $x\preccurlyeq y$ if for any neighborhood $U$ of $y$, there exists a neighborhood $V$ of $x$ such that $U$ contains a homeomorphic copy of $V$. We say that $x$ and $y$ are \emph{equivalent}, or \emph{of the same type}, if $x\preccurlyeq y$ and $y\preccurlyeq x$ and denote this by $x \sim y$. The relation $\sim$ is an equivalence relation on $E$. 
\end{definition}

For $x\in E$, we let $E(x) = \{y\in E \mid y\sim x\}$ denote the equivalence class of $x$. From this, it follows that the relation $\prec$ defined by $E(x) \prec E(y)$ if $x\preccurlyeq y$ and $x\not\sim y$ is a partial order on the set of equivalence classes of ends. 

A point $x\in E$ is said to be \emph{maximal} if $E(x)$ is maximal with respect to the partial order. Denote the set of maximal elements in $E$ by $\mathcal M(E)$. Moreover, if $A$ is a clopen subset of $E$, we let $\mathcal{M}(A)$ denote the set of elements of $A$ which are maximal among all elements of $A$. In \cite{ManRaf20}, Mann--Rafi show the following with regard to this partial order.

\begin{proposition}[Mann--Rafi \cite{ManRaf20}, Proposition 4.7] 
Given an end space $(E, E^G)$, the set of maximal elements, $\mathcal M(E)$, with respect to the partial order $\prec$ is non-empty. Furthermore, for each $x\in \mathcal{M}(E)$, the equivalence class $E(x)$ is either finite or a Cantor set. 
\end{proposition}

\begin{remark}  \label{R:isolated}
Note that the set of isolated planar ends of $S$ forms a single equivalence class in the partial order. If this set is finite, then all of the isolated planar ends are maximal in the partial order. However, if this set is infinite, then the isolated planar ends accumulate onto some other end $x\in E$. In this case, $y \prec x$ for any isolated planar end $y$, and so no isolated planar end is maximal with respect to the partial order.
\end{remark}

Now that we have established the basic properties of the partial order on the space of ends, we need to introduce the notion of ``self-similarity'' in order to state the Mann--Rafi classification of locally CB mapping class groups.

\begin{definition} 
The pair $(E,E^G)$ is \emph{self-similar} if for any decomposition $E=E_1\sqcup E_2 \sqcup \ldots \sqcup E_n$ into pairwise disjoint clopen sets, there exists a clopen subset $D$ of some $E_i$ such that the pair $(D,D\cap E^G)$ is homeomorphic to $(E,E^G)$.
\end{definition}

\noindent Note that this definition also applies to any clopen subset of $E$.

\begin{remark}\label{R:maximal}
We note that when $A$ is a self-similar, clopen subset of $E$, $\mathcal{M}(A)$ is either a single point or a Cantor set of points all of the same type (see \cite[Proposition~4.8]{ManRaf20}). 
\end{remark}

Related to the notion of self-similarity is the following notion of a  ``stable'' neighborhood, which one can think of as a local version of self-similarity.

\begin{definition} 
Let $x$ be an end of $S$. A neighborhood $U$ of $x$ in $E$ is called \emph{stable} if for any smaller neighborhood $V$ of $x$ in $E$, $V$ contains a homeomorphic copy of $U$.
\end{definition}

Finally, we define the following open neighborhoods of the identity in $\Map(S)$. By a subsurface of $S$, we mean a connected, embedded, essential surface $K \subset S$. We call $K\subset S$ \emph{essential} if none of the boundary components of $K$ bound disks in $S$. Notably, we allow boundary components of $K$ to bound once-punctured disks. This usage is somewhat non-standard but makes the classification theorem of Mann--Rafi easier to state. 

\begin{definition}
If $K\subset S$ is a finite-type subsurface of $S$, we define \[\mathcal{V}_K=\{g\in \Map(S) \mid g \text{ restricts to the identity on } K\}.\]
\end{definition}

\noindent The sets $\mathcal{V}_K$ are open neighborhoods of the identity in $\Map(S)$ equipped the compact-open topology. Moreover, these sets form a neighborhood basis at the identity. We can now state the classification theorem for locally CB mapping class groups of Mann--Rafi \cite{ManRaf20}.

\begin{theorem}[Mann--Rafi \cite{ManRaf20}, Theorem 5.7]\label{MannRafiTh5.7}
The group $\Map(S)$ is locally CB if and only if there is a finite-type subsurface $K$ such that the complementary subsurfaces of $K$ each have one or infinitely many ends, 
have genus 0 or infinity, and partition $E$ into finitely many clopen subsets \[E=\left( \bigsqcup_{A\in \mathcal A} A \right) \sqcup \left(\bigsqcup_{P\in \mathcal P} P\right)\] with the following properties.
\begin{enumerate}
    \item Each $A \in \mathcal A$ is self-similar, $\mathcal M(A)\subset \mathcal M(E)$, and $\mathcal M(E)=\bigsqcup_{A\in \mathcal A}\mathcal M(A)$.
    \item Each $P\in \mathcal P$ is homeomorphic to a clopen subset of some $A\in \mathcal A$.
    \item For any $x_A\in \mathcal M(A)$ and any neighborhood $V$ of the end $x_A$ in $S$, there is $f_V\in \operatorname{Homeo}(S)$ so that $f_V(V)$ contains the complementary region to $K$ with end set $A$.
\end{enumerate}

\noindent Moreover, in this case, $\mathcal{V}_K$ is a CB neighborhood of the identity, and $K$ may always be taken to have genus zero if $S$ has infinite genus and genus equal to that of $S$ otherwise.
\end{theorem}

\section{Proof of \Cref{T:BoundedCommutatorLength}}\label{S:ProofOf1.3}

In this section, we prove \Cref{T:BoundedCommutatorLength}, which we use to prove \Cref{T:SCLContinuous} in the next section. Recall \Cref{R:maximal} which tells us that when $A$ is a clopen subset of $E$ that is self-similar, $\mathcal{M}(A)$ consists either of a single point or a Cantor set of ends of the same type. We will need the following lemma about sets of the second kind. 

\begin{lemma}\label{L:SelfSimilarSubsetsHomeo}
Let $A$ be a clopen subset of $E(S)$ which is self-similar and contains a Cantor set of maximal ends (necessarily of the same type). Then, for any decomposition $A = A_1 \sqcup A_2$ into disjoint, clopen subsets such that each $A_i$ contains some elements of $\mathcal{M}(A)$, we have that $A_i$ is homeomorphic to $A$, for $i = 1,2$.
\end{lemma}

\begin{proof}
Let $A$, $A_1$, and $A_2$ be as in the statement of the lemma. 
For each $i = 1, 2$, fix a maximal end $x_i\in \mathcal{M}(A_i)$. The proof of Lemma 5.6 in \cite{ManRaf20} shows that for some $i$, $A_i$ is a stable neighborhood of $x_i$. 
Since $\mathcal{M}(A)$ consists of ends which are all of the same type and $x_i\in \mathcal{M}(A)$, we have that $\mathcal{M}(A) \subset E(x_i)$.
So, by Lemma 4.17 of \cite{ManRaf20}, this implies that every maximal point of $A$ has a stable neighborhood. 

We will first show that $A_1$ is homeomorphic to $A$. To that end, fix a stable neighborhood $U$ of $x_1$, which we know exists from the discussion in the previous paragraph. As sub-neighborhoods of stable neighborhoods are stable (by definition of stability), we may choose our stable neighborhood $U$ so that $U\subset A_1$. Since $\mathcal M(A) \subset E(x_1)$, we have that $y \prec x_1$ or $y \sim x_1$ for all $y\in A_2$. So, $x_1$ is an accumulation point of $E(y)$ by definition of the partial order. 
Thus, Lemma~4.18 of \cite{ManRaf20} tells us that for each $y\in A_2$, we may choose a sufficiently small neighborhood $V_y\subset A_2$ of $y$ so that $U \cup W \cong U$ for any neighborhood $W \subseteq V_{y}$ of $y$. These sets $V_y$ form a cover of $A_2$, and since $A_2$ is compact, there must be a finite sub-cover $\{V_{y_1}, V_{y_2}, \ldots V_{y_k}\}$ of $A_2$. We now will alter these sets $V_{y_i}$ so that they are disjoint. See \Cref{stable_nbhds} for reference.

\begin{figure}[h]
\centering 
\def\svgwidth{0.6\textwidth}
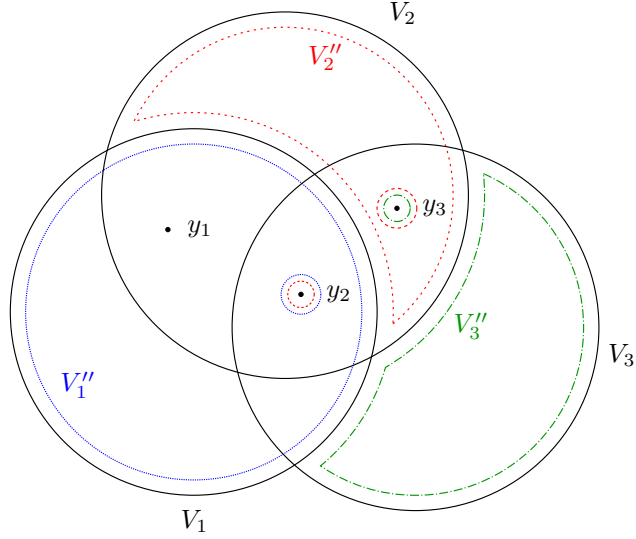
\caption{The neighborhoods $V_i''$. The disjoint neighborhoods $W_i$ are indicated by small circles. A pair of nested circles indicates that a neighborhood $W_i$ has been excised from some $V_j$ and added to $V_i$.}
\label{stable_nbhds}
\end{figure}

Denote the set $V_{y_i}$ simply by $V_i$ for ease of notation. We note that the sets $V_i$ may intersect and, in particular, $y_i$ may lie in a set $V_j$ for some $j\neq i$. To remedy this, we choose smaller neighborhoods $W_i$ of $y_i$ satisfying:
\begin{itemize}
    \item[(a)] the sets $W_i$ are pairwise disjoint, and
    \item[(b)] each set $W_i$ is contained in $V_i$.
\end{itemize}

Notice that the sets \[V_i'=V_i\setminus \bigcup_{j\neq i} W_j\] still cover $A_2$, although they may not be disjoint. Moreover, $V_i'$ contains $y_i$ and does not contain $y_j$ for $j\neq i$. Finally then, the sets \[V_i''=V_i' \setminus \bigcup_{j< i} V_j'\] are disjoint, clopen, and cover $A_2$. Moreover, $V_i''$ is a neighborhood of $y_i$ contained in $V_i$. Thus, we have $U\cup V_i''\cong U$ for each $i$. We use this inductively as follows:

\begin{align*}
    U\sqcup A_2 &= U \sqcup \left(V_1'' \sqcup \cdots \sqcup V_k'' \right) \\
    &\cong (U \sqcup V_1'') \sqcup (V_2'' \sqcup \cdots \sqcup V_k'') \\
    &\cong U \sqcup (V_2'' \sqcup \cdots \sqcup V_k'') \\
    &\ \vdots \\
    &\cong U \sqcup V_k'' \\
    &\cong U.
\end{align*}
Since we chose the stable neighborhood $U$ such that $U \subseteq A_1$, we have $A_1 = (A_1 \setminus U) \sqcup U$. Therefore, \[ A \cong A_1 \sqcup A_2 \cong (A_1 \setminus U) \sqcup (U \sqcup A_2) \cong (A_1 \setminus U) \sqcup U \cong A_1.\] 
Hence, we have shown that $A_1$ is homeomorphic to $A$. By repeating the argument above interchanging the roles of $A_1$ and $A_2$, we see that $A_2$ is also homeomorphic to $A$, as desired.
\end{proof}

A consequence of the proof above is: 

\begin{corollary}\label{C:PinA}
Let $A$ be a clopen subset of $E(S)$ which is self-similar and contains a Cantor set of maximal ends (necessarily of the same type). If $P$ is homeomorphic to a clopen subset of $A$, then $A\sqcup P\cong A$.
\end{corollary}

We note that this corollary follows by letting $A$ play the role of $A_1$ and $P$ play the role of $A_2$ in the proof of \Cref{L:SelfSimilarSubsetsHomeo}. In proving \Cref{T:BoundedCommutatorLength}, we will need to define a particular homeomorphism of $S$ that will be a \emph{shift map}. We use the definition by Abbott, Miller, and the second author in \cite{AbbottMillerPatel}, and note that shift maps have been defined similarly in \cite{ManRaf20}. 

\begin{definition}\label{D:shift}
Let $S'$ be the surface defined by taking the strip $\mathbb{R}\times [-1,1]$, removing a closed disk of radius $\frac{1}{2}$ with center $(n, 0)$ for $n \in \mathbb{Z}$, and attaching any fixed surface with exactly one boundary component to the boundary of each such disk. A \emph{shift} on $S'$ is the homeomorphism that acts like a translation, sending $(x,y)$ to $(x +1, y)$ for $y\in [-1+\epsilon, 1-\epsilon]$ and which tapers to the identity on $\partial S'$. See \Cref{fig:ShiftMap} for an example of such a map. 
\end{definition}

\begin{figure}
    \centering
    \includegraphics[width = 9cm]{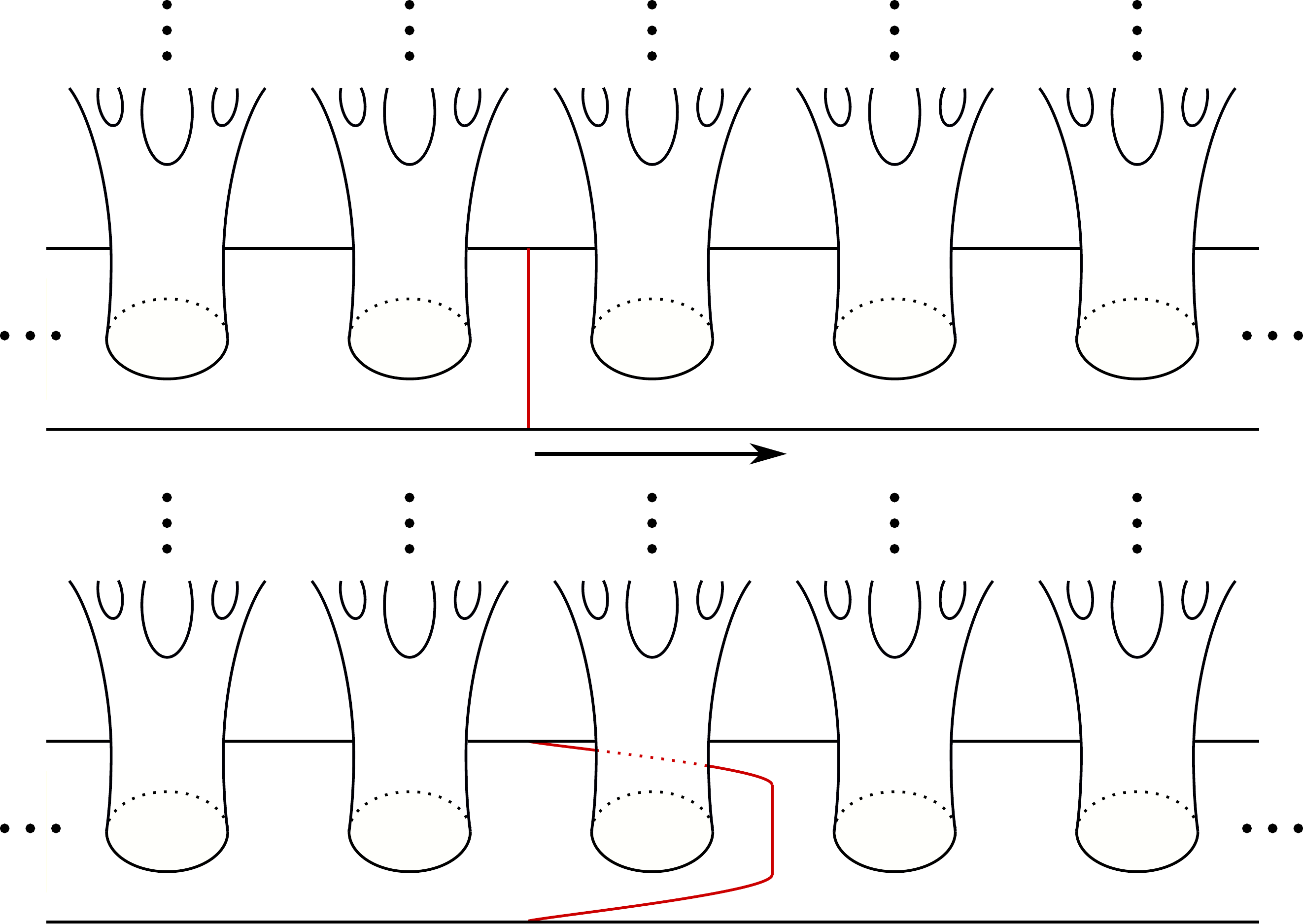}
    \caption{A \emph{shift} on a surface $S'$.}
    \label{fig:ShiftMap}
\end{figure}

Given a surface $S$ with a proper embedding of $S'$ into $S$ so that the two ends of the strip (associated to $+\infty$ and $-\infty$ in $\mathbb{R}$) correspond to two different ends of $S$, the shift on $S'$ induces a shift on $S$, where the homeomorphism acts as the identity on $S\setminus S'$. The support of the shift map on $S$ is the embedded copy of $S'$ in $S$.

\begin{convention*}
For any clopen subset $E'\subset E$, we say that a simple closed curve $\gamma$ \emph{cuts off a subsurface $V$ with end space $E'$} if $S\setminus \gamma$ has two connected components, one with end space $E'$ and the other with end space $E \setminus E'$, and the subsurface $V$ is the closure of the component with end space $E'$. 
Such a curve exists because $E'$ is clopen (see \Cref{S:ends} for details on the topology on $E$). We may assume that $V$ has zero genus if $E' \cap E^G = \emptyset$ by slightly editing our choice of $\gamma$. Therefore, we will always assume that such a curve $\gamma$ cuts off a subsurface $V$ which has zero or infinite genus. 

\end{convention*}

\begin{lemma}\label{ShiftMapTake2}

Let $A$ be a clopen subset of $E(S)$ which is self-similar and contains a Cantor set of maximal ends (necessarily of the same type). Let $V_0$ be any subsurface of $S$ with one boundary component, zero or infinite genus, and for which $E(V_0)$ is a clopen subset of $A$ containing some, but not all, maximal ends of $A$. Then, there exists a countable collection of disjoint subsurfaces $\{V_i\}_{i \in \mathbb{Z}}$ of $S$ and a homeomorphism $f$ of $S$ satisfying:
\begin{itemize}
    \item[(a)] for each $i$, the end space $E(V_i)$ is homeomorphic to $A$; and
    \item[(b)] $f(V_i)=V_{i+1}$ for each $i$.
\end{itemize}
\end{lemma}

\begin{proof}

We will show that the desired homeomorphism $f$ is a shift map on $S$, as in \Cref{D:shift}. Set $V_0$ to be any subsurface as in the statement of the lemma (i.e. a subsurface with $E(V_0)$ containing some, but not all, maximal ends of $A$, with one boundary component, and with zero or infinite genus). Set $A_0=E(V_0)\subset A$. By applying \Cref{L:SelfSimilarSubsetsHomeo} to the partition $A_0 \sqcup (A\setminus A_0)$, we have that $A_0$ and $A\setminus A_0$ are homeomorphic to $A$. Applying \Cref{L:SelfSimilarSubsetsHomeo} again to $A\setminus A_0$, we may partition $A\setminus A_0$ into two clopen subsets, $B_0$ and $C_0$, each containing some maximal ends of $A$, and so that $B_0$ and $C_0$ are both homeomorphic to $A$. So, we have $A=A_0\sqcup B_0\sqcup C_0$, where each set in the union is homeomorphic to $A$.

Fix maximal ends $b\in \mathcal{M}(B_0)$ and $c\in \mathcal{M}(C_0)$. Next, choose partitions $B_0 = B_1^1 \sqcup B_1^2$ and $C_0 = C_1^1 \sqcup C_1^2$ into disjoint clopen sets, each containing points of $\mathcal{M}(A)$, and such that $b\in B_1^1$ and $c\in C_1^1$. Most importantly, \Cref{L:SelfSimilarSubsetsHomeo} applied to $B_0$ and $C_0$ tells us $B_1^1$, $B_1^2$, $C_1^1$, and $C_1^2$ are each homeomorphic to $A$, since $B_0$ and $C_0$ are homeomorphic to $A$. Let $\alpha_{-1}$ be a curve which cuts off a subsurface $V_{-1}$ with end space $B_1^2$, and let $\alpha_1$ be a curve which cuts off a subsurface $V_1$ with end space $C_1^2$. Note that we can take each of these curves to be disjoint from $\alpha_0$ and each other. We now choose partitions $B_1^1 = B_2^1 \sqcup B_2^2$ and $C_1^1 = C_2^1 \sqcup C_2^2$ into disjoint, clopen sets, each containing points of $\mathcal{M}(A)$, and such that $b\in B_2^1$ and $c\in C_2^1$. We let $\alpha_{-2}$ be a curve which cuts off a subsurface $V_{-2}$ with end space $B_2^2$, and let $\alpha_2$ be a curve which cuts off a subsurface $V_{2}$ with end space $C_2^2$.

\begin{figure}[h]
    \centering
    \includegraphics[width = 13cm]{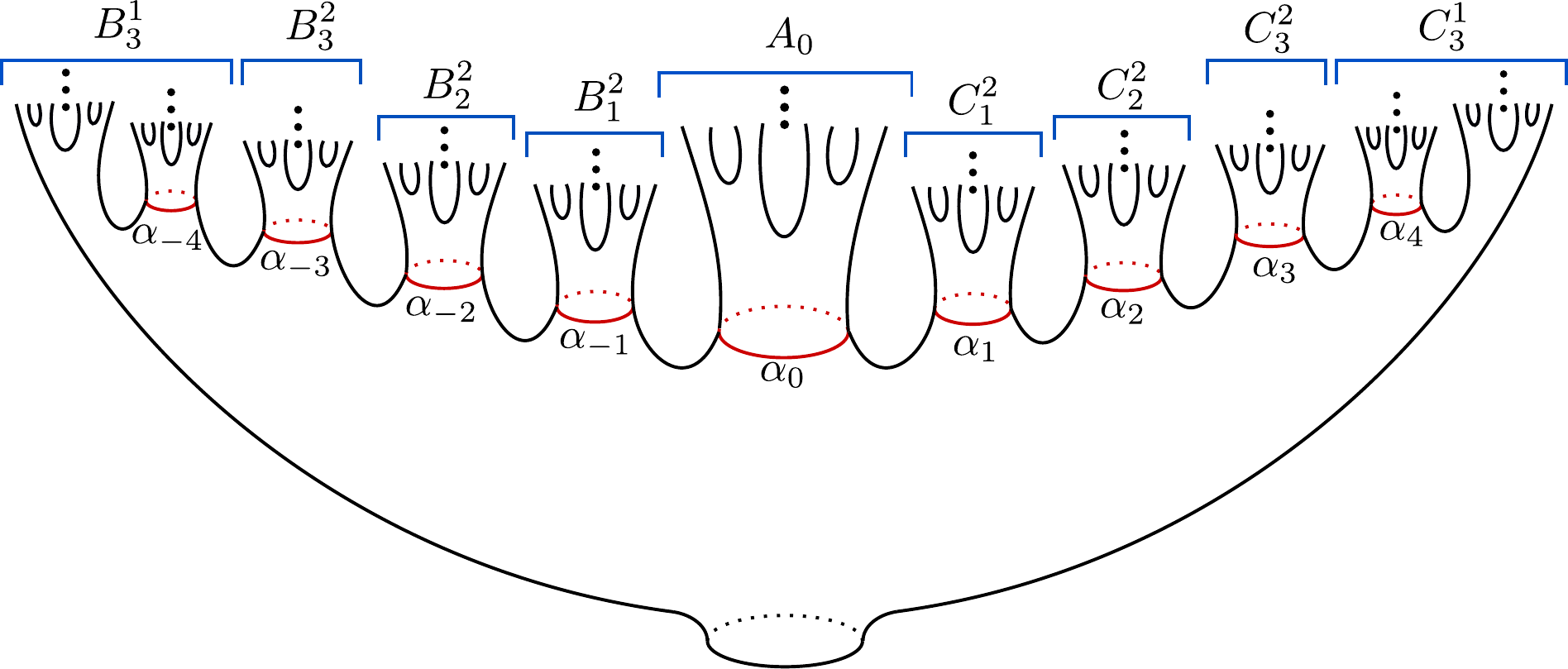}
    \caption{The countable collection of disjoint, homeomorphic subsurfaces of a surface with a self-similar end space and a Cantor set of maximal ends.}
    \label{fig:ShiftMapOnSurface}
\end{figure}

Continuing in this fashion, we choose partitions $B_i^1 = B_{i+1}^1 \sqcup B_{i+1}^2$ and $C_i^1 = C_{i+1}^1 \sqcup C_{i+1}^2$ into disjoint, clopen sets, each containing points of $\mathcal{M}(A)$, and such that $b\in B_{i+1}^1$ and $c\in C_{i+1}^1$. Note that by applying \Cref{L:SelfSimilarSubsetsHomeo} at each stage in this process, we have that $B_i^1$, $B_i^2$, $C_i^1$, and $C_i^2$ are each homeomorphic to $A$ for all $i$. We let $\alpha_{-(i+1)}$ denote a curve which cuts off a subsurface $V_{-(i+1)}$ with end space $B_{i+1}^2$, and let $\alpha_{i+1}$ denote a curve which cuts off a subsurface $V_{i+1}$ with end space $C_{i+1}^2$. See \Cref{fig:ShiftMapOnSurface} for an illustration of this process. By construction, the end space of each $V_i$ is homeomorphic to $A$, the genus of $V_i$ for all $i$ is either zero or infinite, and each $V_i$ has exactly one boundary component, $\alpha_i$, which is disjoint from the boundary of $V_j$ for all $j\neq i$. In particular, the surfaces $V_i$ are all homeomorphic by the classification of surfaces. 

We now construct the desired homeomorphism $f$, which will be a shift map on $S$. The surfaces $\{V_i\}$ correspond to a countable collection of disjoint surfaces with one boundary component of the same topological type, which we call $V$. Let $S'$ be the surface obtained as in \Cref{D:shift} by taking the strip $\mathbb{R}\times [-1,1]$, removing a closed disk of radius $\frac{1}{2}$ with center $(n, 0)$ for each $n \in \mathbb{Z}$, and attaching a copy of $V$ to the boundary of each such disk. Note that there is a proper embedding of $S'$ into $S$ where the boundary of the disk with center $(n, 0)$ is mapped homeomorphically onto $\alpha_n$ for each $n \in \mathbb{Z}$. Lastly, we need to show that the sets $B_i^1$ and $C_i^1$ can be chosen so that $\bigcap_i B_i^1 = \{b\}$ and $\bigcap_i C_i^1 = \{c\}$, respectively, so that the ends $-\infty$ and $+\infty$ of $\mathbb{R} \times [-1,1]$ correspond to the distinct ends $b$ and $c$, respectively. 
We achieve this by putting a metric on the set $A$ so that the sets $B_i^1$ and $C_i^1$ shrink in diameter to zero as $i\to \infty$. Since $B_{i+1}^2$ is chosen to be contained in $B_i^1$ and $C_{i+1}^2$ is chosen to be contained in $C_i^1$ for all $i$, we have that $\bigcap_i B_i^1 = \{b\}$ and $\bigcap_i C_i^1 = \{c\}$, as desired.

Thus, the desired map is the homeomorphism $f$ that acts as a shift on the embedding of $S'$ into $S$ described above and extends via the identity on the rest of $S$. By construction, $f(V_i) = V_{i+1}$ for all $i \in \mathbb{Z}$.
\end{proof}

Recall that Mann--Rafi show that when $S$ is an infinite-type surface such that $\Map(S)$ is locally CB, there exists a finite-type subsurface $K\subset S$ such that the complementary components of $K$ partition $E(S)$ into finitely many clopen sets \[E(S) = A_1 \sqcup \cdots \sqcup A_n \sqcup P_1 \sqcup \cdots \sqcup P_m, \] where each $A_i$ is \emph{self-similar} and each $P_k$ is homeomorphic to a clopen subset of some $A_i$. Using this partition, we prove the following result.

\begin{figure}
    \centering
    \includegraphics[width = 7cm]{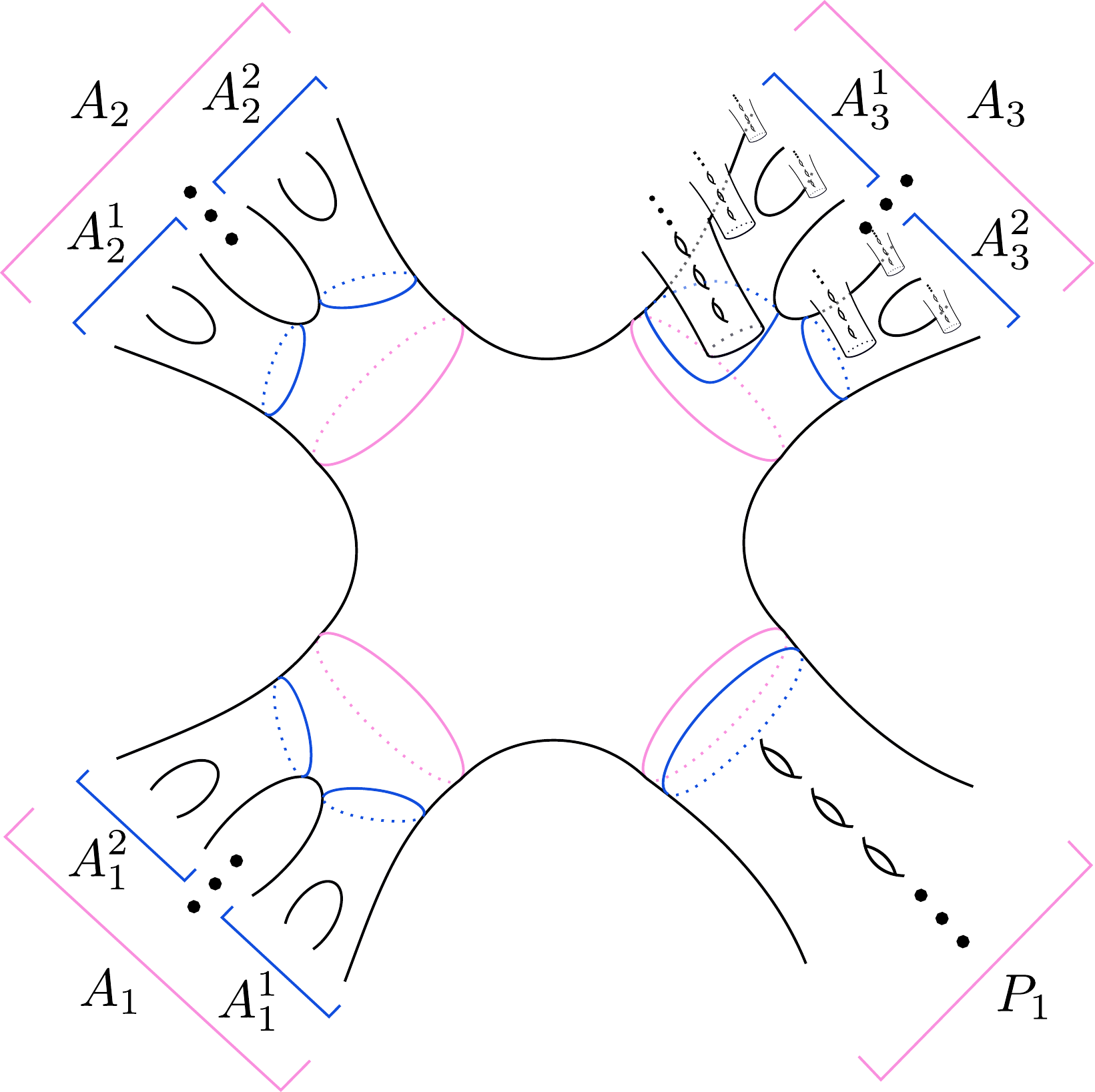}
    \caption{The subsurface $K$ which partitions the end space as in \Cref{MannRafiTh5.7} is bounded by the light pink curves. The extended subsurface $K_1$ which partitions the end space as in the proof of \Cref{T:BoundedCommutatorLength} is bounded by the dark blue curves.}
    \label{fig:EndPartition}
\end{figure}

\medskip

\noindent{\bf \Cref{T:BoundedCommutatorLength}}
    {\em \tBoundedCommLength }

\begin{proof}
Let $S$ and $K$ be as above. 
Note that by \Cref{R:isolated}, if $S$ has finitely many isolated planar ends, then some of the components $A_i$ may consist of a single isolated planar end. 
So, we suppose that the numbering has been chosen so that $A_{n'+1},\ldots,A_n$ each consist of a single isolated planar end (in the case where there are finitely many such ends) and $A_1,\ldots,A_{n'}$ do not. We note that unless $S$ has finitely many isolated planar ends, we have that $n'=n$; see Remark \ref{R:isolated}. However, we do not use this fact in the following proof. By our assumption that each equivalence class of maximal ends is infinite, other than the equivalence class consisting of the isolated planar ends, each set $A_1,\ldots, A_{n'}$ is self-similar with a Cantor set of maximal ends.

We may now enlarge $K$ to a finite-type subsurface $K_1$ so that for each self-similar set $A_i$ with $1\leq i \leq n'$, $K_1$ partitions $A_i$ into two clopen sets, $A_i = A_i^1 \sqcup A_i^2$, where each $A_i^j$ contains points of $\mathcal M(A_i)$ for $j = 1, 2$; see \Cref{fig:EndPartition}. By \Cref{L:SelfSimilarSubsetsHomeo}, both $A_i^1$ and $A_i^2$ are homeomorphic to $A_i$, and thus are each self-similar. The complement of $K_1$ in $S$ then partitions $E(S)$ into \[E(S) =  A_1^1\sqcup A_1^2\sqcup \cdots A_{n'}^1 \sqcup A_{n'}^2 \sqcup A_{n'+1} \sqcup \cdots \sqcup A_n \sqcup P_1 \sqcup \cdots \sqcup P_m \] 
where each $A_i^j$ contains some maximal ends of $A_i$ for $1\leq i\leq n'$. Note that we may choose $K_1$ so that $S\setminus K_1$ consists of subsurfaces $V_i^j$ for $i=1,\ldots,n'$ and $j=1,2$, $V_i$ for $i=n'+1,\ldots n$, and $W_k$ for $k=1,\ldots, m$, each with one boundary component and end space $A_i^j$ (with $1\leq i\leq n'$), $A_i$ (with $n'+1\leq i\leq n$), or $P_k$, respectively. As above, we assume the boundary components of $K_1$ have been chosen appropriately so that each of these subsurfaces has genus zero or infinity. 

Recall that $\mathcal{V}_{K_1} = \{g\in \Map(S) \mid g \text{ restricts to the identity on } K_1\}$. We will show that any element of $\Map(S)$ contained in $\mathcal{V}_{K_1}$ can be written as a product of commutators. Since each of the subsurfaces $V_{n'+1},\ldots,V_n$ is simply a once-punctured disk, an element $h\in \mathcal{V}_{K_1}$ satisfies \[h=\left(\prod_{\substack{1\leq i\leq n' \\ 1\leq j\leq 2}}g_i^j\right)\left(\prod_{1\leq k\leq m} p_k\right)\] 
where the $g_i^j$ and $p_k$ are pairwise-commuting mapping classes supported on $V_i^j$ and $W_k$, respectively. Hence, it suffices to show that each $g_i^j$ and each $p_k$ is a commutator in $\Map(S)$.

\medskip

\noindent \textbf{Claim 1:} Each $g_i^j$ is a commutator.

\begin{proof}[Proof of Claim 1]
We will prove this for $g = g_1^1$ and note that the proof for $g_i^j$ is the same for $1\leq i \leq n'$ and $1\leq j \leq 2$. Note that $A_1$ is a clopen subset of $E(S)$ which is self-similar and contains a Cantor set of maximal ends. Setting $V_0=V_1^1$, we note that the end space $A_1$ and the subsurface $V_0$ satisfy the hypotheses of \Cref{ShiftMapTake2}. So, by \Cref{ShiftMapTake2} we may extend $V_0$ to a countable collection of disjoint subsurfaces $\{V_i\}_{i\in \Z}$ of $S$ equipped with a homeomorphism $u$ of $S$ satisfying $u(V_i)=V_{i+1}$ for each $i\in \mathbb{Z}$. Define a homeomorphism $v$ by:
\begin{equation*}
v(x)=\begin{cases}
         \, u^i g u^{-i}(x) \quad &\text{if } \, x \in V_i \text{ for some } i\geq 0; \text{ and}\\
     \, x \quad &\text{otherwise. }
     \end{cases}
\end{equation*}

Note that $v$ is a well-defined homeomorphism since the $V_i$ surfaces are disjoint. We claim that $g = [v, u] = v u v^{-1} u^{-1}$, i.e. $[v,u]$ acts as $g$ on $V_0$ and as the identity everywhere else. Indeed, note that $\supp(v) \subseteq \bigcup_{i=0}^\infty V_i$. So, if $x\not\in \bigcup_{i = 0}^\infty V_i$, then $vuv^{-1}u^{-1}(x) = x$. If $x\in V_0$, then $u^{-1}(x) \in V_{-1}$. As $v$ acts as the identity on $V_{-1}$, we have that $vuv^{-1}u^{-1}(x) = v(x)= g(x)$, since $v(x) = g(x)$ on $V_0$ by definition. So, suppose $x\in V_i$ for some $i> 0$. Then, \begin{align*}
    vuv^{-1}u^{-1}(x) &= v u (u^{i-1} g^{-1} u^{-i + 1}) u^{-1}(x) \\
    &= (u^i g u^{-i}) u (u^{i-1} g^{-1} u^{-i + 1}) u^{-1}(x) \\
    &= x.
\end{align*}
Thus, $g = [v, u]$, as claimed.
\end{proof}

\medskip

\noindent \textbf{Claim 2:} In fact, fixing $j$, $\prod_{1\leq i\leq n'}g_i^j$ is a single commutator. 

\begin{proof}[Proof of Claim 2]
Denote the elements constructed in Claim 1 as $u_i^j, v_i^j$ so that $[v_i^j, u_i^j] = g_i^j$. Note that $u_i^j$ and $v_i^j$ are supported on the component of $S\setminus K$ with end space $A_i$ so that, as we vary $i$ with $j$ fixed, $u_i^j$ and $v_i^j$ commute with both $u_{i'}^j$ and $v_{i'}^j$ whenever $i \neq i'$. Thus, $$\prod_{1\leq i\leq n'}g_i^j =\left[\prod_{1\leq i\leq n'}v_i^j, \prod_{1\leq i\leq n'}u_i^j\right]$$ is a single commutator.

\end{proof}

\medskip

\noindent \textbf{Claim 3:} Each $p_k$ is a commutator.

\begin{proof}[Proof of Claim 3]
We show this explicitly for the case $p=p_1$, the other cases $1\leq k\leq m$ being identical. To do this, we will show that $p$ is conjugate to a mapping class supported in some $V_i^j$.

Note that $P_1$ is homeomorphic to a clopen subset of some $A_i$ with $1\leq i\leq n'$. Without loss of generality, we assume that $P_1$ is homeomorphic to a clopen subset of $A_1$. Since $A_1^1\cong A_1$, we have that $P_1$ is in turn homeomorphic to a clopen subset $Q_1$ of $A_1^1$. We claim that $E(S) \setminus P_1$ and $E(S)\setminus Q_1$ are also homeomorphic. To see this, first note that 
\[E(S)\setminus P_1 =A_1^1 \sqcup A_1^2 \sqcup \left(\bigsqcup_{2\leq i\leq n} A_i\right) \sqcup \left( \bigsqcup_{k\geq 2} P_k\right)\]
and 
\[E(S)\setminus Q_1 = (A_1^1 \setminus Q_1) \sqcup A_1^2 \sqcup P_1  \sqcup \left(\bigsqcup_{2\leq i \leq n} A_i\right) \sqcup \left(\bigsqcup_{k\geq 2} P_k \right). \]
Since \[(A_1^1 \setminus Q_1)\sqcup A_1^2 \sqcup P_1 \cong (A_1^1 \setminus Q_1)\sqcup A_1^2 \sqcup Q_1 = A_1^1\sqcup A_1^2,\] we see that $E(S)\setminus P_1$ and $E(S)\setminus Q_1$ are homeomorphic. Recall that $W_1$ is a subsurface of $S$ with one boundary component, of zero or infinite genus, and with end space $P_1$. Given that $Q_1$ is clopen in $A_1^1$, we can choose a subsurface $U_1$ of $V_1^1$ with end space $Q_1$, one boundary component, and genus zero or infinity. Then by the classification of surfaces, $W_1$ and $U_1$ are homeomorphic, and given the argument above, $S\setminus W_1$ and $S \setminus U_1$ are homeomorphic as well. Thus, there is a homeomorphism $w$ of $S$ mapping $W_1$ to $U_1$ so that $wpw^{-1}$ is supported on $U_1\subset V_1^1$. By the proof of Claim 1, any mapping class supported on $V_1^1$ is a commutator. Hence, $wpw^{-1}$, and therefore $p$ itself, is a commutator.
\end{proof}

Now, let $\mathcal{O} = \mathcal V_{K_1} = \{g\in \Map(S) \mid g|_{K_1} = \operatorname{id}\}$, and note that $\mathcal{O}$ is an open subgroup of $\Map(S)$. If $g\in \mathcal{O}$, then we have as above that $g$ is a product of mapping classes $g_i^j$ with $1\leq i \leq n'$ and $1\leq j \leq 2$ and $p_k$ with $1\leq k \leq m$. Claims 1 and 3 show that each $g_i^j$ and each $p_k$ is a commutator, and we have that $\cl(g)\leq 2n'+m\leq 2n+m$. To obtain the better bound $\cl(g)\leq 2$, note that the proof of Lemma A.1 from the Appendix shows that we can take $m=0$ by absorbing each $P_k$ into the appropriate $A_i$, and Claim 2 shows that the product of $g_i^j$ with $1\leq i \leq n'$ and $1\leq j \leq 2$ is, in fact, a product of 2 commutators.  
\end{proof}

\section{Proof of \Cref{T:SCLContinuous}}\label{S:ProofOf1.1}

In this section, let $S$ be an infinite-type surface such that $\Map(S)$ is locally CB and such that every equivalence class of maximal ends of $S$ is infinite, other than a finite set of isolated planar ends. Recall that if $K\subset S$ is a finite-type subsurface, then $\mathcal V_K$ is the open neighborhood of the identity in $\Map(S)$ which consists of those mapping classes which restrict to the identity on $K$. We first prove a result which holds for general topological groups. The proof of this result is based heavily on the proof of \cite[Theorem 1.5]{bhw}.

\begin{proposition}
\label{prop:sclcontinuous}
Let $G$ be a topological group. Suppose that there is a neighborhood $\mathcal O$ of the identity in $G$ which is contained in the commutator subgroup $[G,G]$ and that $\scl(g) = 0$ for all $g\in \mathcal O$. Then:
\begin{itemize}
    \item[(a)] any homogeneous quasimorphism $\phi:G\to \R$ is continuous; and
    \item[(b)] the stable commutator length function is continuous on $[G,G]$.
\end{itemize}
\end{proposition}

\begin{proof}
We first prove (a). Consider a homogeneous quasimorphism $\phi:G\to \R$. We note that by Bavard Duality (\Cref{BavardDuality}), since $\scl(g) = 0$ for all $g\in \mathcal O$, we must have that $\phi(g) = 0$ for all $g\in \mathcal O$.

We will now prove that any homogeneous quasimorphism $\phi:G\to \R$ is continuous on $[G,G]$. To that end, choose any integer $k \geq 1$ and let $g\in [G,G]$. Fix a neighborhood $\mathcal{U}_k$ of $g$ small enough so that if $h\in \mathcal U_k$, then $g^kh^{-k} \in \mathcal O$ (in particular we may choose $\mathcal U_k$ to be the pre-image of $\mathcal O$ under the continuous map $G\to G$ defined by $h\mapsto g^kh^{-k}$). Let $\phi:G\to \R$ be a homogeneous quasimorphism. If $h \in \mathcal U_k$ as above, then we have \[k|\phi(g)-\phi(h)|=|\phi(g^k)+\phi(h^{-k})-\phi(g^kh^{-k})|\leq D(\phi).\] Here, the equality holds by the definition of a homogeneous quasimorphism and the fact that $\phi(g^kh^{-k})=0$ (since $g^kh^{-k} \in \mathcal O$). Hence, for all $k>0$, there exists a neighborhood $\mathcal U_k$ of $g$ such that $|\phi(g)-\phi(h)|\leq \frac{D(\phi)}{k}$ for all $h \in \mathcal U_k$. This implies that $\phi$ is continuous.

Finally, we prove continuity of $\scl$ as a function on $[G,G]$. Choose $g\in [G,G]$, let $k\geq 1$ be any integer, and define the neighborhood $\mathcal U_k$ of $g$, depending on $k$, as before. Choose $\phi$ to be any homogeneous quasimorphism on $G$ with $D(\phi)=1$. By the above paragraph, if $h\in \mathcal U_k$ then we have \[\frac{|\phi(h)|-\frac{1}{k}}{2}\leq \frac{|\phi(g)|}{2} \leq \frac{|\phi(h)|+\frac{1}{k}}{2}.\] Taking the supremum over all homogeneous quasimorphisms $\phi$ of defect 1 yields \[\scl(h)-\frac{1}{2k}\leq \scl(g)\leq \scl(h)+\frac{1}{2k}\] for all $h\in \mathcal U_k$. This proves that $\scl$ is continuous.
\end{proof}

As a consequence of \Cref{T:BoundedCommutatorLength}, we now prove our main result about stable commutator length and the commutator subgroup. We note that by \Cref{prop:sclcontinuous}, the following result holds whenever there exists an open subgroup $\mathcal{O}$ in $\Map(S)$ and a constant $B > 0$ such that if $f\in \mathcal{O}$, then $f$ is a product of at most $B$ commutators in $\Map(S)$. The proof of \Cref{T:BoundedCommutatorLength} shows this condition holds for surfaces $S$ satisfying the conditions above using the neighborhood $\mathcal{V}_{K_1}$, but it may also hold for a larger class of surfaces.

\medskip

\smallskip

\noindent{\bf \Cref{T:SCLContinuous}}
    {\em \tSCLisContinuous } 
    
\begin{proof}
We let $K$ be the subsurface of $S$ which partitions $E(S)$ as in the statement of \Cref{MannRafiTh5.7}. 
As in the proof of \Cref{T:BoundedCommutatorLength}, we let $K_1\subset S$ be an enlargement of $K$ such that the complementary components to $K_1$ in $S$ partition $E(S)$ into finitely many clopen sets \[ E(S) = A_1^1 \sqcup A_1^2 \sqcup \cdots \sqcup A_{n'}^1 \sqcup A_{n'}^2 \sqcup A_{n'+1} \sqcup \cdots \sqcup A_n \sqcup P_1 \sqcup \cdots \sqcup P_m, \] where each of the $A_i^j$ contain some maximal ends of $S$. 
Note that by the proof of \Cref{T:BoundedCommutatorLength}, the commutator subgroup $[\Map(S), \Map(S)]$ contains the open subset $\mathcal{V}_{K_1}$. Therefore, by \Cref{L:OpenClosedSubgroup}, $[\Map(S), \Map(S)]$ is both an open and closed subgroup of $\Map(S)$. Recall that $\Map(S)$ is a Polish group, i.e., it is separable and completely metrizable. As any closed subspace of a Polish space is Polish, we have that $[\Map(S), \Map(S)]$ is a Polish group as well. 

Using the open subgroup $\mathcal V_{K_1}$, the rest of the theorem now follows from \Cref{prop:sclcontinuous} and \Cref{T:BoundedCommutatorLength}. 
\end{proof}    

Given a surface $S$ as above, we also show that the commutator length, $\cl$, on the commutator subgroup varies continuously in a certain coarse sense.

\begin{corollary}
There exists $B >0$ such that if $f\in [\Map(S),\Map(S)]$ then there is a neighborhood $\mathcal U$ of $f$ in $\Map(S)$ with the property that if $g\in \mathcal U$, then $g\in [\Map(S),\Map(S)]$ and \[|\cl(g)-\cl(f)|\leq B.\]
\end{corollary}

\begin{proof}
Let $f\in [\Map(S),\Map(S)]$. Choose a neighborhood $\mathcal U$ of $f$ in $\Map(S)$ small enough so that if $g\in \mathcal U$, then $f^{-1}g \in \mathcal V_{K_1}$ (for instance choose $\mathcal U$ to be the coset $f \mathcal{V}_{K_1}$).

Set $f^{-1}g = h$, where $h \in \mathcal{V}_{K_1}$, and note that $h$ is a product of at most $2n +m$ commutators by \Cref{T:BoundedCommutatorLength}. Then, $g = fh$, and so $\cl(g)\leq \cl(f)+2n+m$. Similarly, $f = g h^{-1}$, and so $\cl(f)\leq \cl(g)+2n+m$. Letting $B = 2n +m$ gives the result.

\end{proof}

With some additional assumptions on the surface $S$, we can prove that the abelianization of $\Map(S)$ is finitely generated. We will assume that our surface has \emph{tame} end space and that $\Map(S)$ is generated by a coarsely bounded set. 
In particular, we will be restricting to surfaces $S$ which satisfy the hypotheses of \cite[Theorem 1.6]{ManRaf20}. For ease of exposition, we omit the definition of tame end space and refer the reader to \cite{ManRaf20} for details.

\medskip

\smallskip

\noindent{\bf \Cref{T:FiniteAbelianization}}
    {\em \tFiniteAbelianization }

\begin{proof}
Recall that $\mathcal{V}_K = \{g\in \Map(S) \mid g|_K = \operatorname{id}\}$. 
In \cite[Theorem 1.6]{ManRaf20}, Mann--Rafi show that 
the group $\Map(S)$ is generated by $\mathcal{V}_K$ and a finite set $F$ (see discussion before the proof of Theorem 1.6 in \cite{ManRaf20}), where $K$ is any subsurface satisfying the conditions in \Cref{MannRafiTh5.7}. Note that the subsurface $K_1$, which is an extension of the subsurface $K$, also satisfies the hypotheses of \Cref{MannRafiTh5.7}. Indeed, condition (1) follows from our assumption that $K_1$ partitions each $A\in \mathcal{A}$ containing a Cantor set of maximal ends into two clopen subsets which are again self-similar, condition (2) is immediate from the definition of self-similarity, and condition (3) follows since the complementary regions of $K$ contain the complementary regions of $K_1$. 
Hence, $\Map(S)$ is generated by $\mathcal{V}_{K_1}$ and another finite set $F_1$.
As $\mathcal{V}_{K_1}\subset [\Map(S),\Map(S)]$ by \Cref{T:BoundedCommutatorLength}, the abelianization of $\Map(S)$ is generated by the image of $F_1$, and is thus finitely generated. 

Note that the quotient of a topological group by an open subgroup is discrete; see, for instance, \cite[2.6.6 (iii)]{HigginsTopGroups} (or simply note that the preimage of a point in the quotient is a coset of the given open subgroup). By \Cref{T:SCLContinuous}, $[\Map(S), \Map(S)]$ is an open subgroup of $\Map(S)$, and hence the abelianization of $\Map(S)$ is discrete. 

\end{proof}

\section*{Appendix}\label{Appendix}
In this appendix, we give a concrete topological characterization of the surfaces that satisfy the hypotheses of \Cref{T:SCLContinuous} and \Cref{T:BoundedCommutatorLength}. The following lemma was suggested to us by J. Malestein and J. Tao, who also outlined the proof that we give below. As per Malestein--Tao \cite{selfsim}, we call a surface $S$ \emph{uniformly self-similar} if $S$ has zero or infinite genus and $E(S)$ is self-similar and contains a Cantor set of maximal ends. In this case, we will say that $E(S)$ is a uniformly self-similar set. 

\medskip

\noindent{\bf Lemma A.1}
    {\em \tTopologicalCharacterization }

\begin{proof}
Suppose first that $S$ is the connected sum of finitely many uniformly self-similar surfaces, 
with a finite-type surface $T$. If $T$ has $k$ punctures, we can express $T$ as a connected sum of a closed surface $T'$ with $k$ once-punctured spheres. Label the uniformly self-similar surfaces $S_1, \ldots, S_n$ and let $A_i$ denote the endspace of the surface $S_i$. Letting $K = T'$, the complementary components of $K$ partition the endspace of $S$ as $E(S) = A_1 \sqcup \cdots \sqcup A_n \sqcup A_{n+1} \sqcup \cdots \sqcup A_{n+k}$, where the sets $A_i$ for $i = n+1, \ldots, n+k$ are single isolated punctures.
We note that since each $S_i$ is uniformly self-similar, each endspace $A_i$ contains a Cantor set of maximal ends for $i = 1, \ldots, n$. 
Thus, we may enlarge $K$ to a finite-type subsurface $K_1$ so that for each uniformly self-similar set $A_i$, $K_1$ partitions $A_i$ into two clopen sets, $A_i = A_i^1 \sqcup A_i^2$, where each $A_i^j$ contains points of $\mathcal M(A_i)$ for $j = 1, 2$. We further assume that $K_1$ is such that each boundary curve is separating. We let $V_i^j$ denote the complementary component of $S - K_1$ with endspace $A_i^j$ for $1\leq i\leq n$ and $1\leq j \leq 2$. 

We note that by assumption and construction, $S$ may possibly have finitely many isolated planar ends (if $T$ has finitely many isolated punctures), and otherwise each equivalence class of maximal ends is infinite. So, it remains to prove that $\Map(S)$ is locally CB. To do so, we will use the characterization of locally CB mapping class groups of Mann--Rafi \cite{ManRaf20}. The subsurface $K_1$ partitions the endspace of $S$ into finitely many clopen sets $E(S) = A_1^1 \sqcup A_1^2 \sqcup \cdots \sqcup A_n^1 \sqcup A_n^2 \sqcup A_{n+1} \sqcup \cdots \sqcup A_{n+k}$ where each $A_i^j$ and $A_r$ (for $r= n+1, \ldots, n+k$) is self-similar, and so this partition satisfies property (1) of \Cref{MannRafiTh5.7}. As there are no $P_i$'s in this construction, property (2) is vacuously satisfied. So, it remains to prove condition (3).

To that end, let $x$ be any maximal end. If $x$ is the unique end in $A_r$ for $r = n+1, \ldots, n+k$, the property is immediately satisfied letting $f_V$ be the identity homeomorphism. Now let $x \in \mathcal{M}(A_i^j)$. Up to relabeling, we may assume, without loss of generality, that $x\in \mathcal{M}(A_1^1)$. Let $V$ be any neighborhood of $x$. 
By shrinking $V$ slightly, we may assume that $V \subset V_1^1$ and has one boundary component. As in the proof of \cite[Theorem 5.7]{ManRaf20}, we have that the pair $(V, S - V)$ is homeomorphic to the pair $(V_1^1, S - V_1^1)$. Indeed, by \Cref{L:SelfSimilarSubsetsHomeo}, we have that $A_1^2$ is homeomorphic to $A_1^2 \sqcup (A_1^1 - E(V))$. Hence $E(S-V)=E(S-V_1^1)\cup E(V_1^1-V)\cong E(S-V_1^1)$. Note also that $E(V) \cong E(V_1^1)$ by \Cref{L:SelfSimilarSubsetsHomeo}. Note that if $S_1$ has infinite genus, then so do $V$ and $V_1^1$, as do $S - V$ and $S - V_1^1$. If $S_1$ has zero genus, then so do $V$ and $V_1^1$; and $S - V$ and $S - V_1^1$ have the same genus as $S$. Thus, by the classification of surfaces, there is a homeomorphism $f_V$ of $S$ taking $V$ to $V_1^1$, as desired.

For the other direction, suppose that $S$ is locally CB and that every equivalence class of maximal ends of $S$ is infinite, except for possibly the isolated punctures. We choose $K$ to partition $E(S)$ as in \Cref{MannRafiTh5.7}. In other words, $K$ partitions $E$ into $A_1\sqcup \cdots \sqcup A_n \sqcup A_{n+1} \sqcup \cdots \sqcup A_{n+k}$, collectively containing the maximal ends of $E$, and finitely many sets $P_1\sqcup \ldots \sqcup P_m$ homeomorphic to clopen subsets of the sets $A_i$. 
Here, $A_1, \ldots, A_n$ are uniformly self-similar and $A_{n+1}, \ldots, A_{n+k}$ are isolated punctures. As usual, we may suppose that the complementary subsurfaces to $K$ each have one boundary component and genus zero or infinity. Denote by $V_i$ the subsurface with endspace $A_i$ and $W_j$ the subsurface with endspace $P_j$. Note that by capping off the boundary component of $V_i$ with a disk, we obtain a self-similar surface.
Without loss of generality, we assume that $m>0$ and $P_1$ is homeomorphic to a clopen subset of $A_1$. We may replace $K$ by a subsurface $K'$ obtained as follows: join $\partial V_1$ to $\partial W_1$ via a simple arc $\gamma$ contained in $K$, and define $K'$ to be the complement in $K$ of a small regular neighborhood of $V_1\cup W_1\cup \gamma$. Denote by $V_1'$ this regular neighborhood of $V_1\cup W_1\cup \gamma$. Then the complementary subsurfaces to $K'$ are $V_1',V_2,\ldots,V_n$ and $W_2,W_3,\ldots,W_m$. 
We claim that $V_1'$ is a uniformly self-similar surface minus an open disk. 
To see this, note first that $V_1'$ has genus zero or infinity and one boundary component. Furthermore, $V_1'$ has endspace $A_1\sqcup P_1$ which is homeomorphic to $A_1$ by \Cref{C:PinA}. This proves the claim. 
We may remove the remaining clopen sets $P_j$ in the partition one by one by absorbing them into the relevant $A_i$ in this way. The result is a subsurface $K''\subset K$ partitioning $E(S)$ into self-similar sets $A_1'',\ldots,A_n'', A_{n+1}, \ldots, A_{n+k}$ where for each $i$, the set $A_i''$ could simply be $A_i$, and where these sets collectively contain the maximal ends of $S$. Additionally, the complementary subsurface $V_i''$ to $K$ with endspace $A_i''$ has one boundary component and genus zero or infinity. Thus, by capping off the boundary component of $V_i''$ with a disk, we obtain a uniformly self-similar surface, and $S$ is the connected sum of these uniformly self-similar surfaces together with a finite-type surface obtained from $K'' \cup V_{n+1} \cup \cdots \cup V_{n+k}$ by capping off the boundary components corresponding to $V_1'', \ldots, V_n''$. This completes the proof.
\end{proof}

We note that in the proof of \Cref{T:BoundedCommutatorLength}, Claim 2 is not actually necessary since we show in this appendix that the $P_k$ sets can be absorbed into the self-similar sets $A_i$. We included Claim 2 in the proof of that theorem for completeness and for future applications.

\bibliographystyle{plain}
  \bibliography{biblio}

\end{document}